\documentclass[12pt,reqno]{article}

\usepackage[usenames]{color}
\usepackage{amssymb}
\usepackage{graphicx}
\usepackage{amscd}

\usepackage[colorlinks=true,
linkcolor=webgreen,
filecolor=webbrown,
citecolor=webgreen]{hyperref}

\definecolor{webgreen}{rgb}{0,.5,0}
\definecolor{webbrown}{rgb}{.6,0,0}

\usepackage{color}
\usepackage{fullpage}
\usepackage{float}

\usepackage{graphics,amsmath,amssymb}
\usepackage{amsthm}
\usepackage{amsfonts}
\usepackage{latexsym}
\usepackage{epsf}

\newcommand{\seqnum}[1]{\href{http://oeis.org/#1}{\underline{#1}}}

\theoremstyle{plain}
\newtheorem{theorem}{Theorem}
\newtheorem{corollary}[theorem]{Corollary}
\newtheorem{lemma}[theorem]{Lemma}

\theoremstyle{definition}

\newtheorem{conjecture}[theorem]{Conjecture}

\theoremstyle{remark}
\newtheorem{remark}[theorem]{Remark}

\def\modd#1 #2{#1\ \mbox{\rm (mod}\ #2\mbox{\rm )}}

\def\Enn{{\mathbb{N}}}

\def\codelink{\url{https://cs.uwaterloo.ca/~shallit/papers.html}}

\title{Lagrange's Theorem for Binary Squares}

\author{P.~Madhusudan\thanks{This material is based upon work supported by the National Science Foundation under Grant No.~1527395}\\
Department of Computer Science\\
Thomas M. Siebel Center for Computer Science\\
201 North Goodwin Avenue\\
Urbana, IL 61801-2302 \\
USA \\ {\tt madhu@illinois.edu} \\
\and
Dirk Nowotka \\ Department of Computer Science \\
Kiel University \\
D-24098 Kiel \\
Germany \\
{\tt dn@informatik.uni-kiel.de} \\
\and
Aayush Rajasekaran and Jeffrey Shallit \\
School of Computer Science \\
University of Waterloo \\
Waterloo, ON  N2L 3G1 \\
Canada \\
{\tt arajasekaran@uwaterloo.ca, shallit@uwaterloo.ca}
}

\begin{document}

\theoremstyle{plain}

\theoremstyle{definition}
\newtheorem{openproblem}[theorem]{Open Problem}

\maketitle

\begin{abstract}
We show how to prove theorems in additive number theory using a decision procedure based on finite automata.  Among other things, we obtain the following
analogue of Lagrange's theorem:  
every natural number $> 686$ is the sum of at most $4$ natural numbers
whose canonical base-$2$ representation is a {\it binary square}, that is, a string
of the form $xx$ for some block of bits $x$.   Here
the number $4$ is optimal.
While we cannot embed this theorem itself in a decidable
theory, we show that \emph{stronger} lemmas that imply
the theorem can be embedded in decidable theories, and show how automated methods can be used to search for these stronger lemmas.
 \end{abstract}

\section{Introduction}
\label{intro}
Additive number theory is the study of the additive properties of
integers \cite{Nathanson:1996}.  In particular, an {\it additive basis of order $h$} is 
a subset $S \subseteq \Enn$ such that every natural number is the
sum of $h$ members, not necessarily distinct, of $S$.
The principal problem of additive number theory is to determine
whether a given subset $S$ is an additive basis of order $h$ for some $h$,
and if so, to determine the smallest value of $h$.   There has been much research in the area, and deep techniques, such as the Hardy-Littlewood circle method, have been developed to solve these kinds of problems \cite{Vaughan:1997}.

One of the earliest results in additive number
theory is Lagrange's famous
theorem \cite{Lagrange:1770} that every natural number is the sum of four squares \cite{Grosswald:1985,Moreno&Wagstaff:2005}.  In the 
terminology of the previous paragraph, this means that
$S = \{ 0^2, 1^2, 2^2, 3^2,\ldots \}$ forms an additive basis
of order $4$.  
The celebrated problem of Waring (1770)
(see, e.g., \cite{Ellison:1971,Small:1977,Vaughan&Wooley:2002}) is to determine
the corresponding least order $g(k)$ for $k$'th powers. 
Since it is easy to see that
numbers of the form $4^a(8k+7)$ cannot be expressed as the sum of
three squares, it follows that $g(2) = 4$.   It is known that
$g(3) = 9$ and $g(4) = 19$.

In a variation on this concept we say that $S \subseteq \Enn$ is an
{\it asymptotic additive basis of order $h$} if every {\it sufficiently
large\/} natural number is the sum of $h$ members, not necessarily
distinct, of $S$.  The classical function $G(k)$ is defined to be the least
asymptotic basis order for $k$'th powers.  From above we have $G(2) = 4$.
It is known that $G(14) = 16$, and $4 \leq G(3) \leq 7$. 
Despite much work, the exact value of $G(3)$ is currently
unknown.

Recently there has been interest in doing additive number theory on sets of natural numbers whose
base-$k$ representations have certain properties.  For example, 
Banks \cite{Banks:2016} proved that every natural number is the sum of at most $49$ natural numbers whose decimal representation is a palindrome.  This was improved by
Cilleruelo, Luca, and Baxter \cite{Cilleruelo&Luca&Baxter:2017} to $3$ summands for every base $b \geq 5$.
The remaining cases $b = 2,3,4$ were recently resolved \cite{Rajasekaran&Shallit&Smith:2018}.

In this paper we consider a variation on Lagrange's theorem.  Instead
of the ordinary notion of the square of an integer, we consider
``squares'' in the sense of formal language theory \cite{Hopcroft&Ullman:1979}.   That is,
we consider $x$, the canonical binary (base-$2$) representation of an integer
$N$, and call $N$ a {\it binary square} if $N = 0$, or if
$x = yy$ for some nonempty string $y$ that starts with a $1$.
Thus, for example, $N = 221$ is a binary square, since
$221$ in base $2$ is $11011101 = (1101)(1101)$.  The first
few binary squares are
$$ 0,3,10,15,36,45,54,63,136,153,170,187,204,221,238,255, \ldots;$$
they form sequence \seqnum{A020330} in the {\it On-Line Encyclopedia
of Integer Sequences} (OEIS) \cite{oeis}.   Clearly a number
$N > 0$ is a binary square if and only if it is of the
form $a (2^n + 1)$ for $n \geq 1$ and
$2^{n-1} \leq a < 2^n$.  
This is a very natural sequence to
study, since
the binary squares have density $\Theta(N^{1/2})$
in the natural numbers, just like the ordinary squares.  
(There exist sets of density $\Theta(N^{1/2})$ that do not form an asymptotic basis of finite order, so density considerations alone do {\it not\/} imply our result.)

In this paper we prove the following result.
\begin{theorem}
The binary squares form an asymptotic basis of order $4$.
More precisely,
every natural number $N>686$ is the sum of $4$
binary squares.  There are $56$ exceptions, given below:
\begin{multline*} 1,2,4,5,7,8,11,14,17,22,27,29,32,34,37,41,44,47,53,62,95,104,107,113,116,122,125,\\ 131,134,140,143,148,155,158,160,167,407,424,441,458,475,492,509,526,552,560, \\
   \ \ \ \ \     569,587,599,608,613,620,638,653,671,686. \hfill\qed
\end{multline*}
\label{main}
\end{theorem}
The novelty in our approach is that we obtain this theorem in additive number theory {\it using very little number theory at all}.
Instead, we use an approach based on formal language theory, reducing the proof of the
theorem to a decidable language emptiness problem.
Previously we obtained similar results for palindromes
\cite{Rajasekaran:2018,Rajasekaran&Shallit&Smith:2017,Rajasekaran&Shallit&Smith:2018}. 

\subsection{Search for appropriate lemmas and proving the theorem}

The technique we use for encoding Theorem~\ref{main} as a problem dealing with automata is to ask, for all sufficiently large integers $N$, whether there exist four binary squares with
representation $x_i x_i$, $1 \leq i \leq 4$, such that the sum of the numbers they represent is $N$. Since the language of binary squares is not regular, we use an encoding where we represent only
one copy of each $x_i$ and reuse it to represent the number. However, it turns out that we cannot represent the desired theorem directly as an emptiness/universality problem of finite automata. The reason is that when representing only one copy of the $x_i$, we can do ``school addition'' (aligning them and adding the numbers, columnwise, with a carry) only if the words $x_i$ are roughly of the same \emph{length}. More precisely, we require the lengths of the squares employed to either be bounded by a constant, or differ from each other and from the number $N$ only by a bounded length.

For fixed constants $k_i$, $1 \leq i \leq 4$,
we observe that the set of all binary representations
of $N$ for which
there exist four words $x_i$, $1 \leq i \leq 4$, of lengths $L-k_i$, 
such that the binary representation of $N$ is of length
$2L$ and the sum of the numbers represented by $x_i x_i$, $1 \leq i \leq 4$ is $N$, is a regular language.
Thus we can prove, using known decision algorithms for automata, lemmas that assert that all numbers of a particular form can be represented by a sum of four binary squares, where the binary squares are of various lengths $L-k_1$, $L-k_2$, $L-k_3$, and $L-k_4$, for a finite set of tuples $\langle k_1, k_2, k_3, k_4 \rangle$ (see Lemma~5 for such a lemma).

Proving such a lemma for a particular set of combinations of lengths implies the theorem, of course, but the lemma itself is stronger. The truth of such stronger lemmas is decidable, while we don't have a way to directly decide the theorem itself! 

Thus we need a \emph{search} for an appropriate lemma for a particular combination of length differences that is valid. Given that checking these lemmas for any set of combinations is decidable, we can do the search for these lemmas automatically. We tried various combinations and succeeded in proving one lemma, namely Lemma~5, that implies our theorem.

The above technique can be generalized to some extent--- evidently, we could also consider the analogous results for other powers
such as cubes, and bases $b \geq 2$,
but we do not do that in this paper.  

\subsection{Notation}

We are concerned with the binary representation of numbers, so let us
introduce some notation.  If $N$ is a natural number, then by $(N)_2$ we mean the string giving the canonical base-$2$ representation of $N$, having no leading zeroes.
For example, $(43)_2 = 101011$.  The canonical representation of $0$ is $\epsilon$, the empty string.

If $2^{n-1} \leq N < 2^n$ for $n \geq 1$, we say that
$N$ is an {\it $n$-bit integer\/} in base $2$.  
Note that the first bit of the binary representation of an $n$-bit integer
is always nonzero. 
The {\it length\/} of an integer $N$ satisfying
$2^{n-1} \leq N < 2^n$ is defined to be $n$; alternatively, the
length of $N$ is $1 + \lfloor \log_2 N \rfloor$.
For $n \geq 1$ we define 
$C_n = \{ a \cdot (2^n + 1)  \ : \ 
2^{n-1} \leq a < 2^n \},$
the set of all $2n$-bit binary squares.

\section{A classical approach}

In this section we describe how one can apply classical
number-theoretic and combinatorial tools to this problem to obtain some results
weaker than Theorem~\ref{main}.
The idea is to show that the numbers that are the sum
of two binary squares form a set of positive 
lower asymptotic density.  
(In contrast, our approach via automata, which we discuss in
later sections, provides more
precise results.)

For sets $S, T \subseteq \Enn$ we define the sumset
$S + T = \{ s+t \ : \ s \in S,\ t \in T \}.$
The cardinality of a finite set $S$ is denoted by
$|S|$.
Given a set $S \subseteq \Enn$, the {\it lower
asymptotic density} of $S$ is defined to be
$$ d(S) = \liminf_{n \rightarrow \infty} 
\frac{| \{x \in S \ : \ 1 \leq x \leq n \} |}{n} .$$

We first prove

\begin{lemma}
For $n \geq 1$ we have $|C_n + C_{n+1}| = 2^{2n-1}$.
\end{lemma}

\begin{proof}
Since $|C_n| = 2^{n-1}$ and $|C_{n+1}| = 2^n$,
this lemma is equivalent to the claim that each member
of the sumset $C_n + C_{n+1}$ has a 
{\it unique} representation
as the sum of one element of $C_n$ and one element of
$C_{n+1}$.

We argue by contradiction.  Suppose the representation
is not unique, and there exist
integers $a, a'$ with $2^{n-1} \leq a, a' < 2^n - 1$
and integers $b, b'$ with $2^n \leq b, b' <2^{n+1}$
such that $(a,a') \not= (b,b')$ but
\begin{equation}
a \cdot (2^n + 1) + b \cdot (2^{n+1} + 1)
= a' \cdot (2^n + 1) + b' \cdot (2^{n+1}+1) .
\label{ab}
\end{equation}
Computing Eq.~\eqref{ab} modulo $2^n + 1$, we see that
$-b \equiv \modd{-b'} {2^n + 1}$.  Since 
$2^n \leq b, b' <2^{n+1}$ we see the congruence
in fact implies that $b = b'$.
But then $a = a'$, a contradiction.
\end{proof}

\begin{theorem}
The numbers that are the sum of two binary squares
form a set of lower asymptotic density 
$\geq 1/40$.
\end{theorem}

\begin{proof}
Let $S_2$ be the set of numbers that are the sum of two
binary squares.  Clearly $C_n + C_{n+1} \subseteq S_2$.

There are $2^{2n-1}$ elements in the sumset
$C_n + C_{n+1}$, whose largest element is 
$(2^n - 1)2^n + (2^{n+1} - 1)2^{n+1} =
5 \cdot 2^{2n} - 3 \cdot 2^n$.
Given an integer $m\geq 14$, choose $n \geq 1$ such that
$5 \cdot 2^{2n} - 3 \cdot 2^n \leq
	m < 5 \cdot 2^{2n+2} - 3 \cdot 2^{n+1} $.
Then 
$$ \frac{| \{ x \in S_2 \ : \ 1 \leq x \leq m \}|}{m}
\geq \frac{2^{2n-1}}{5 \cdot 2^{2n+2}} = \frac{1}{40}.$$
\end{proof}

\begin{corollary}
The binary squares form an asymptotic basis of finite
order.
\end{corollary}

\begin{proof}
This is a direct consequence of a result of Nathanson 
\cite[Theorem 11.6, p.~366]{Nathanson:2000},
which says that if a subset $S$ of $\Enn$ has $0 \in S$,
$\gcd(S) = 1$, and has positive lower asymptotic density,
then it is an asymptotic basis of finite order.   It is now easy
to check that the hypotheses are fulfilled for $S = S_2$.
\end{proof}

\begin{remark}
It would be interesting to determine the exact lower
asymptotic density of the set $S_2$.  
Numerical computation suggests that perhaps
$d(S_2) \doteq .14$.
\end{remark}

\section{The automaton approach:  the main lemma}

Now we turn to a completely different approach to 
the theorem for binary squares, as sketched in Section~\ref{intro}, 
using automata theory.  This allows us to obtain the upper
bound $4$ for the number of binary squares, a stronger result
than obtained using the classical approach.

Our main lemma is
\begin{lemma}
\leavevmode
\begin{itemize}
\item[(a)] Every length-$n$ integer, $n$ odd, $n \geq 13$, is the sum of binary squares as follows: either
\begin{itemize}
\item one of length $n-1$ and one of length $n-3$, or 
\item two of length $n-1$ and one of length $n-3$, or
\item one of length $n-1$ and two of length $n-3$, or
\item one each of lengths $n-1$, $n-3$, and $n-5$, or
\item two of length $n-1$ and two of length $n-3$, or
\item two of length $n-1$, one of length $n-3$, and
one of length $n-5$.
\end{itemize}
\item[(b)]  Every length-$n$ integer, $n$ even, $n \geq 18$, is the sum of binary squares as follows:  either
\begin{itemize}
\item two of length $n-2$ and two of length $n-4$, or
\item three of length $n-2$ and one of length $n-4$, or
\item one each of lengths $n$, $n-4$, and $n-6$, or
\item two of length $n-2$, one of length $n-4$, and one of length $n-6$.
\end{itemize}
\end{itemize}
\label{aayush}
\end{lemma}
Lemma~\ref{aayush} almost immediately proves
Theorem~\ref{main}:
\begin{proof}
If $N<2^{17} = 131072$, the result can be proved by a completely straightforward computation using dynamic programming:   to form the sumset $S \oplus T$, given finite sets of natural numbers $S$ and $T$, we use a bit vector corresponding to the elements of $S$,
and then take its XOR shifted by each element of $T$.    When we do this, we find that 
there are 
\begin{itemize}
\item 256 binary squares $< 2^{17}$;
\item 19542 numbers $<2^{17}$ that are the sum of two binary squares;
\item 95422 numbers $<2^{17}$ that are the sum of three binary squares;
\item 131016 numbers $<2^{17}$ that are the sum of four binary squares.
\end{itemize}
Otherwise $N \geq 2^{17}$, so $(N)_2$ is a binary string of length $n \geq 18$.   If $n$ is odd, the result follows from Lemma~\ref{aayush} (a).  If $n$ is even, the result follows from Lemma~\ref{aayush} (b).
\end{proof}
It now remains to prove Lemma~\ref{aayush}.   We do this in the next section.

\section{Proof of Lemma~\ref{aayush}}

In this section we prove Lemma~\ref{aayush} in detail.
\begin{proof}
The basic idea is to use nondeterministic finite automata (NFAs).  These are finite-state machines where each input corresponds to multiple computational paths; an input is accepted iff {\it some\/} computational path leads to a final state.  We assume the reader is familiar with the basics of this theory; if not, please consult, e.g., \cite{Hopcroft&Ullman:1979}.  For us, an NFA is a quintuple $(Q, \Sigma, \delta, q_0, F)$, where $Q$ is the set of states,
$\Sigma$ is the input alphabet, $\delta$ is the transition function,
$q_0$ is the initial state, and $F$ is the set of final states.

We
construct an NFA that, on input an integer $N$ written in binary, ``guesses''
a representation as a sum of binary squares, and then verifies that
the sum is indeed $N$.  Everything is done using a reversed representation,
with least significant digits processed
first.  There are some complications, however.  

First, with an NFA we cannot verify that a guessed string is indeed a binary square,
as the language $\{ xx \ : \ x \in 1\{0,1\}^* \}$ is not a regular language.
So instead we only guess the ``first half'' of a binary square.
Now, however, we are forced to choose a slightly unusual representation for
$N$, in order to be able to compare the sum of our guessed powers
with the input $N$.  If $N$ were represented in its ordinary base-$2$ representation, this
would be impossible with an NFA, since once we process the guessed ``first half''
and compare it to the input, we would no longer have the ``second half'' (identical
to the first) to compare to the rest of the input.

To get around this problem, we represent integers $N$ in a kind of ``folded
representation'' over the input alphabet $\Sigma_2 \cup (\Sigma_2 \times \Sigma_2)$, where $\Sigma_k = \{ 0,1, \ldots, k-1 \}$. 
The idea is to present our NFA with two bits of the input string at once,
so that we can add both halves of our guessed powers
at the same time, verifying that we are producing $N$ as we go.
Note that we use slightly different representations for the two parts of Lemma~\ref{aayush}. 
The precise representations are detailed in their respective subsections.



We can now prove Lemma~\ref{aayush} by phrasing it 
as a language inclusion problem. For each of the two parts of the lemma,
we can build an NFA $A$ 
that only accepts such folded strings if they repesent numbers
that are 
the sum of any of the combination of squares as described in the lemma.
We also create an NFA, $B$, that accepts
all valid folded representations that are sufficiently long. We then 
check the assertion that the language recognized by $B$ is a subset of that recognized by $A$.

\subsection{Odd-length inputs}

Again, to flag certain positions of the input tape, we use an extended alphabet.
Define $$\Gamma = \{1_f\} \ \cup \ \bigcup_{\alpha \in \{a,b,c,d,e\}} \{[0,0]_\alpha,[0,1]_\alpha,[1,0]_\alpha,[1,1]_\alpha\}.$$

Let $N$ be an integer, and let $n = 2i+1$ be the length
of its binary representation. We write 
$(N)_2 = a_{2i} a_{2i - 1} \cdots a_1 a_0$ and fold this
to produce the input
string $$[a_{i}, a_0]_a [a_{i+1}, a_1]_a  \cdots [a_{2i-5}, a_{i-5}]_a
[a_{2i-4}, a_{i-4}]_b
[a_{2i-3}, a_{i-3}]_c
[a_{2i-2}, a_{i-2}]_d
[a_{2i-1}, a_{i-1}]_e
 a_{{2i}_f}.$$
 
Let $A_{\rm odd}$ be the NFA that recognizes those odd-length integers, represented in this folded
format, that are the sum of binary squares meeting any of the 6 conditions
listed in Lemma~\ref{aayush} (a). We construct $A_{\rm odd}$ as the union of several automata
$A(t_{n-1},t_{n-3}, m_a)$ and $B(t_{n-1},t_{n-3},t_{n-5}, m_b)$.
The parameters $t_{p}$ represent the number of 
summands of length $p$ we are guessing. 
The parameters $m_a$ and $m_b$ are the carries that 
we are guessing will be produced by the first half of the summed binary squares. $A$-type 
machines try summands of lengths $n-1$ and $n-3$ only, while $B$-type machines 
include at least one $(n-5)$-length summand.
We note that for the purpose of summing, guessing $t$ binary squares
is equivalent to guessing a {\it single\/} square over the larger alphabet $\Sigma_{t+1}$.

We now consider the construction of a single automaton
$$A(t_{n-1},t_{n-3}, m) = (Q \cup \{q_{\rm acc}, q_0, s_1\}, \Gamma, \delta, q_0, \{q_{\rm acc} \}).$$
The elements of $Q$ have 4 non-negative parameters and are of the form $q(x_1, x_2, c_1, c_2)$. Because the $t_{n-3}$ summand is
not aligned with the input, we use our states to ``remember''
our guesses. When we make a guess at the higher end of the
$t_{n-3}$ summand, it must be used as the guess for its lower
end on the \emph {next step}. We remember this guess by storing 
it as the $x_2$ parameter.
The parameter $x_1 \leq t_{n-3}$ is the last digit of 
the guessed summand of length $n-3$. 
We use $c_1$ to track the higher carry, and $c_2$ to track the
lower carry. We must have $c_1, c_2 < t_{n-1} + t_{n-3}$. 

We now discuss the transition function, $\delta$ of our NFA. In this section,
we say that the sum of natural numbers, $\mu_1$ and $\mu_2$, ``produces'' an output
bit of $\theta \in \Sigma_2$ with a ``carry'' of $\gamma$ if 
$\mu_1 + \mu_2 \equiv \theta \;(\bmod\; 2)$ and
$\gamma = \left \lfloor{\frac{\mu_1 + \mu_2}{2}}\right \rfloor $. 

We allow a transition from $q_0$ 
to $q(x_1, x_2, c_1, c_2)$ on the letter $[j,k]_a$ iff
there exists $0 \leq r \leq t_{n-1}$ such that
$x_2 + r + m$ produces an output of $j$ with a carry of $c_1$ and
$x_1 + r$ produces an output of $k$ with a carry of $c_2$. 

We allow a transition from $q(x_1, x_2, c_1, c_2)$ 
to $q(x_1', x_2', c_1', c_2')$ on the letters $[j,k]_a$ and $[j,k]_b$ iff
there exists $0 \leq r \leq t_{n-1}$ such that
$x_2' + r + c_1$ produces an output of $j$ with a carry of $c_1'$ and
$x_2 + r + c_2$ produces an output of $k$ with a carry of $c_2'$. Elements
of $Q$ have identical transitions on inputs with subscripts $a$ and $b$. The 
reason we have the letters with subscript $b$ is for $B$-machines, which guess a summand of length $n-5$.

There is only one letter of the input with the subscript $c$, and it corresponds to
the last higher guess of the summand of length $n-3$. 
We allow a transition from $q(x_1, x_2, c_1, c_2)$ 
to $q(x_1', t_{n-3}, c_1', c_2')$ on the letter $[j,k]_c$ iff
there exists $0 \leq r \leq t_{n-1}$ such that
$t_{n-3} + r + c_1$ produces an output of $j$ with a carry of $c_1'$ and
$x_2 + r + c_2$ produces an output of $k$ with a carry of $c_2'$.

There is only one letter of the input with the subscript $d$, and it corresponds to
the second-last lower guess of the summand of length $n-3$. 
We allow a transition from $q(x_1, t_{n-3}, c_1, c_2)$ 
to $q(x_1', 0, c_1', c_2')$ on the letter $[j,k]_d$ iff
there exists $0 \leq r \leq t_{n-1}$ such that
$r + c_1$ produces an output of $j$ with a carry of $c_1'$ and
$t_{n-3} + r + c_2$ produces an output of $k$ with a carry of $c_2'$.

There is only one letter of the input with the subscript $e$, and it corresponds to
the last lower guess of the summand of length $n-3$. 
We allow a transition from $q(x_1, 0, c_1, c_2)$ 
to $s_1$ on the letter $[j,k]_e$ iff
$t_{n-1} + c_1$ produces an output of $j$ with a carry of $1$ and
$x_1 + t_{n-1} + c_2$ produces an output of $k$ with a carry of $m$.

Finally, we add a transition from $s_1$ to $q_{\rm acc}$ on the letter $1_f$.

We now consider the construction of a single automaton
$$B(t_{n-1},t_{n-3}, t_{n-5}, m) = (P \cup Q \cup \{q_{\rm acc}, q_0, s_1\}, \Gamma, \delta, q_0, \{q_{\rm acc} \}).$$
The elements of $P$ have 6 non-negative parameters and are of the form 
$q(x_1, x_2, y_1, y_3, c_1, c_2)$.
The parameter $x_1 \leq t_{n-3}$ is the last digit of 
the guessed summand of length $n-3$ and
$x_2 \leq t_{n-3}$ is the previous higher guess of the length-$n-3$ summand.
The parameter $y_1 \leq t_{n-5}$ is the last digit of 
the guessed summand of length $n-5$ and
$y_3 \leq t_{n-5}$ is the previous higher guess of the length-$n-5$ summand.
We use $c_1$ to track the higher carry, and $c_2$ to track the
lower carry. We must have $c_1, c_2 < t_{n-1} + t_{n-3} + t_{n-5}$. 
The elements of $Q$ have 8 non-negative parameters and are of the form 
$$q(x_1, x_2, y_1, y_2, y_3, y_4, c_1, c_2).$$
The parameter $x_1 \leq t_{n-3}$ is the last digit of 
the guessed summand of length $n-3$ and
$x_2 \leq t_{n-3}$ is the previous higher guess of the length-$n-3$ summand.
The parameters $y_1, y_2 \leq t_{n-5}$ are the last digit and the second-last
digit of the guessed summand of length $n-5$ respectively.
The parameter
$y_3, y_4 \leq t_{n-5}$ are the two most recent higher guess of the length-$n-5$ summand,
with $y_4$ being the most recent one.
We use $c_1$ to track the higher carry, and $c_2$ to track the
lower carry. We must have $c_1, c_2 < t_{n-1} + t_{n-3} + t_{n-5}$. 

We now discuss the transition function, $\delta$ of our NFA. 
We allow a transition from $q_0$ 
to $p(x_1, x_2, y_1, y_3, c_1, c_2)$ on the letter $[j,k]_a$ iff
there exists $0 \leq r \leq t_{n-1}$ such that
$x_2 + y_3 + r + m$ produces an output of $j$ with a carry of $c_1$ and
$x_1 + y_1 + r$ produces an output of $k$ with a carry of $c_2$. 

We use a transition from $p(x_1, x_2, y_1, y_3, c_1, c_2)$
to $q(x_1, x_2', y_1, y_2', y_3, y_4', c_1', c_2')$ on the letter $[j,k]_a$ iff
there exists $0 \leq r \leq t_{n-1}$ such that
$x_2' + y_4' + r + c_1$ produces an output of $j$ with a carry of $c_1$ and
$x_2 + y_2' + r + c_2$ produces an output of $k$ with a carry of $c_2$. 

We use a transition from $q(x_1, x_2, y_1, y_2, y_3, y_4, c_1, c_2)$
to $q(x_1, x_2', y_1, y_2, y_4, y_4', c_1', c_2')$ on the letter $[j,k]_a$ iff
there exists $0 \leq r \leq t_{n-1}$ such that
$x_2' + y_4' + r + c_1$ produces an output of $j$ with a carry of $c_1$ and
$x_2 + y_3 + r + c_2$ produces an output of $k$ with a carry of $c_2$. 

We use a transition from $q(x_1, x_2, y_1, y_2, y_3, t_{n-5}, c_1, c_2)$
to $q(x_1, x_2', y_1, y_2, t_{n-5}, t_{n-5}, c_1', c_2')$ on the letter $[j,k]_b$ iff
there exists $0 \leq r \leq t_{n-1}$ such that
$x_2' + r + c_1$ produces an output of $j$ with a carry of $c_1$ and
$x_2 + y_3 + r + c_2$ produces an output of $k$ with a carry of $c_2$. 

We use a transition from $q(x_1, x_2, y_1, y_2, t_{n-5}, t_{n-5}, c_1, c_2)$
to $q(x_1, t_{n-3}, y_1, y_2, t_{n-5}, t_{n-5}, c_1', c_2')$ on the letter $[j,k]_c$ iff
there exists $0 \leq r \leq t_{n-1}$ such that
$t_{n-3} + r + c_1$ produces an output of $j$ with a carry of $c_1$ and
$x_2 + y_3 + r + c_2$ produces an output of $k$ with a carry of $c_2$. 

We use a transition from $q(x_1, t_{n-3}, y_1, y_2, t_{n-5}, t_{n-5}, c_1, c_2)$
to $q(x_1, t_{n-3}, y_1, y_2, t_{n-5}, t_{n-5}, c_1', c_2')$ on the letter $[j,k]_d$ iff
there exists $0 \leq r \leq t_{n-1}$ such that
$r + c_1$ produces an output of $j$ with a carry of $c_1$ and
$t_{n-3} + y_1 + r + c_2$ produces an output of $k$ with a carry of $c_2$. 

We use a transition from $q(x_1, t_{n-3}, y_1, y_2, t_{n-5}, t_{n-5}, c_1, c_2)$
to $s_1$ on the letter $[j,k]_e$ iff
$t_{n-1} + c_1$ produces an output of $j$ with a carry of $1$ and
$x_1 + y_2 + t_{n-1} + c_2$ produces an output of $k$ with a carry of $m$. 

Finally, we add a transition from $s_1$ to $q_{\rm acc}$ on the letter $1_f$.

We now turn to verification of the inclusion assertion.
We used the Automata Library toolchain of the {\tt ULTIMATE} program analysis 
framework \cite{Heizmann&co:2013,Heizmann&co:2016} to establish our results. The {\tt ULTIMATE}
code proving our result can be found in the file $\tt OddSquareConjecture.ats$
at \codelink.\\ Since the constructed machines get very large, 
we wrote a C++ program generating these machines, which can be found 
in the file
$\tt OddSquares.cpp$ at  the same location.

The final machine, $A_{\rm odd}$, has 2258 states. The syntax checker, $B$, has 8 states. 
We then asserted that the language recognized by $B$ is a subset of that recognized by $A$. 
ULTIMATE verified this assertion in under a minute. 
Since this test succeeded, the proof of Lemma~\ref{aayush} (a) is
complete. 

\subsection{Even-length inputs}

In order to flag certain positions of the input tape, we use an extended alphabet.
Define $$\Gamma = \left(\bigcup_{\alpha \in \{a,b,c,d,e\}} \{[0,0]_\alpha,[0,1]_\alpha,[1,0]_\alpha,[1,1]_\alpha\} \right)
\cup  \left(\bigcup_{\beta \in \{f,g,h,i\}} \{0_\beta, 1_\beta\} \right).$$ 

Let $N$ be an integer, and let $n = 2i+4$ be the length
of its binary representation. We write 
$(N)_2 = a_{2i+3} a_{2i+2} \cdots a_1 a_0$ and fold this
to produce the input
string $$[a_{i}, a_0]_a [a_{i+1}, a_1]_b [a_{i+2}, a_2]_c [a_{i+3}, a_3]_c \cdots [a_{2i-3}, a_{i-3}]_c
[a_{2i-2}, a_{i-2}]_d
[a_{2i-1}, a_{i-1}]_e a_{2i_f} a_{{2i+1}_g} a_{{2i+2}_h} a_{{2i+3}_i}.$$

Let $A_{\rm even}$ be the NFA that recognizes the even-length integers, represented in this folded
format, iff the integer is the sum of binary squares meeting any of the 4 conditions
listed in Lemma~\ref{aayush} (b). We construct $A_{\rm even}$ as the union of several automata
$A(t_n,t_{n-2},t_{n-4},t_{n-6}, m)$. The parameters $t_{p}$ represent the number of 
summands of length $p$ we are guessing. 
The parameter $m$ is the carry that 
we are guessing will be produced by the first half of the summed binary squares.
Again, guessing $t$ binary squares
is equivalent to guessing a {\it single\/} square over the larger alphabet $\Sigma_{t+1}$.

We now consider the construction of a single automaton
$$A(t_n,t_{n-2},t_{n-4},t_{n-6}, m) = (Q \cup \{q_{\rm acc} \}, \Gamma, \delta, q_0, \{q_{\rm acc} \}).$$
The elements of $Q$ have 8 non-negative parameters and are of the form 
$q(x_1, x_2, x_3, y_1, z_1, z_2, c_1, c_2)$.
The parameter $x_1$ is the second digit of the guessed summand of length $n$. The parameters
$x_2$ and $x_3$ represent the previous 2 lower guesses of the length-$n$ summand; these 
must be the next 2 higher guesses of this summand. The parameter $y_1$ represents
the previous lower guess of the length-$(n-2)$ summand. We set $z_1$ as the last
digit of the guessed summand of length $n-6$, while $z_2$ is the previous higher 
guess of this summand. Finally, $c_1$ tracks the lower carry, while $c_2$ tracks the
higher carry. For any $p$, we must have $x_p \leq t_n$, $y_p \leq t_{n-2}$, $z_p \leq t_{n-6}$, and 
$c_p < t_n + t_{n-2} + t_{n-4} + t_{n-6}$. 
The initial state, $q_0$, is $q(0,0,0,0,0,0,0,0)$.

We now discuss the transition function, $\delta$ of our NFA. Note that in 
our representation of even-length 
integers, the first letter of the input must have the subscript
$a$, and it is the only letter to do so. 
We only allow the initial state to have outgoing transitions on such letters. 

We allow a transition from $q_0$ 
to $q(x_1, 0, x_3, y_1, z_1, z_2, c_1, c_2)$ on the letter $[j,k]_a$ iff
there exists $0 \leq r \leq t_{n-4}$ such that
$x_1 + t_{n-2} + r + z_2 + m$ produces an output of $j$ with a carry of $c_2$ and
$x_3 + y_1 + r + z_1$ produces an output of $k$ with a carry of $c_1$. 

The second letter of the input must have the subscript
$b$, and it is the only letter to do so.
We allow a transition from $q(x_1, 0, x_3, y_1, z_1, z_2, c_1, c_2)$ 
to $q(x_1, x_3, x_3', y_1', z_1, z_2', c_1', c_2')$ on the letter $[j,k]_b$ 
iff
there exists $0 \leq r \leq t_{n-4}$ such that
$t_n + y_1 + r + z_2' + c_2$ produces an output of $j$ with a carry of $c_2'$ and
$x_3' + y_1' + r + z_2 + c_1$ produces an output of $k$ with a carry of $c_1'$. 

We allow a transition from $q(x_1, x_2, x_3, y_1, z_1, z_2, c_1, c_2)$ 
to $q(x_1, x_3, x_3', y_1', z_1, z_2', c_1', c_2')$ on the letter $[j,k]_c$ 
iff
there exists $0 \leq r \leq t_{n-4}$ such that
$x_2 + y_1 + r + z_2' + c_2$ produces an output of $j$ with a carry of $c_2'$ and
$x_3' + y_1' + r + z_2 + c_1$ produces an output of $k$ with a carry of $c_1'$. 

The letter of the input with the subscript
$d$ corresponds to the last guess of the lower half of the summand of length 
$n-6$, and it is the only letter to do so.
We allow a transition from $q(x_1, x_2, x_3, y_1, z_1, t_{n-6}, c_1, c_2)$ 
to $q(x_1, x_3, x_3', y_1', z_1, 0, c_1', c_2')$ on the letter $[j,k]_d$ 
iff
there exists $0 \leq r \leq t_{n-4}$ such that
$x_2 + y_1 + r + c_2$ produces an output of $j$ with a carry of $c_2'$ and
$x_3' + y_1' + r + t_{n-6} + c_1$ produces an output of $k$ with a carry of $c_1'$. 

The letter of the input with the subscript
$e$ corresponds to the last guess of both halves of the summand of length 
$n-4$, and it is the only letter to do so.
We allow a transition from $q(x_1, x_2, x_3, y_1, z_1, 0, c_1, c_2)$ 
to $q(x_1, x_3, x_3', y_1', 0, 0, 0, c_2')$ on the letter $[j,k]_e$ 
iff
$x_2 + y_1 + t_{n-4} + c_2$ produces an output of $j$ with a carry of $c_2'$ and
$x_3' + y_1' + t_{n-4} + z_1 + c_1$ produces an output of $k$ with a carry of $m$. 

We allow a transition from $q(x_1, x_2, x_3, y_1, 0, 0, 0, c_2)$ 
to $q(x_1, x_3, 0, 0, 0, 0, 0, c_2')$ on the letter $j_f$ 
iff
$x_2 + y_1 + c_2$ produces an output of $j$ with a carry of $c_2'$. 

We allow a transition from $q(x_1, x_2, 0, 0, 0, 0, c_2)$ 
to $q(x_1, 0, 0, 0, 0, 0, 0, c_2')$ on the letter $j_g$ 
iff
$x_2 + t_{n-2} + c_2$ produces an output of $j$ with a carry of $c_2'$. 

We allow a transition from $q(x_1, 0, 0, 0, 0, 0, c_2)$ 
to $q(0, 0, 0, 0, 0, 0, 0, c_2')$ on the letter $j_h$ 
iff
$x_1 + c_2$ produces an output of $j$ with a carry of $c_2'$. 

We allow a transition from $q(0, 0, 0, 0, 0, 0, 0, c_2)$ 
to $q_{\rm acc}$ on the letter $1_i$ 
iff
$t_n + c_2$ produces an output of $1$ with a carry of $0$. 

The final machine, $A_{\rm even}$ is constructed as the union of 15 automata:
\begin{itemize}
\item $A(0,2,2,0,m)$, varying $m$ from 0 to 3
\item $A(0,3,1,0,m)$, varying $m$ from 0 to 3
\item $A(1,0,1,1,m)$, varying $m$ from 0 to 2
\item $A(0,2,1,1,m)$, varying $m$ from 0 to 3
\end{itemize}

We now turn to verification of the inclusion assertion.  The {\tt ULTIMATE}
code proving our result can be found in the file $\tt EvenSquareConjecture.ats$
at \codelink.\\ Since the constructed machines get very large, 
we wrote a C++ program generating these machines, which can be found 
in the file
{\tt EvenSquares.cpp} at the same location.

The final machine, $A_{\rm even}$, has 1343 states. The syntax checker, $B$, has 12 states. 
We then asserted that the language recognized by $B$ is a subset of that recognized by $A$. 
ULTIMATE verified this assertion in under a minute. 
Since this test succeeded, the proof of Lemma~\ref{aayush} (b) is
complete. 

\end{proof}

\begin{corollary}
Given an integer $N > 686$, we can find an expression for $N$ as the sum of four binary squares in time linear in
$\log N$.
\end{corollary}
\begin{proof}
For $N < 131072$, we do this with a simple brute-force search via dynamic
programming, as explained previously.
Otherwise we construct the appropriate automaton $A$ 
(depending on whether the binary representation of $N$
has either even or odd length).  
Now carry out the usual direct product construction for intersection of languages on $A$ and $B$, where $B$ is 
the automaton accepting the folded binary representation of
$N$.  The resulting automaton has at most $c\log N$ states
and transitions.  Now use the usual depth-first search
of the transition graph to find a path from the initial
state to a final state.  The labels of this path gives
the desired representation.
\end{proof}

\subsection{Ensuring correctness}

 As in every machine-based proof, we want some assurance that our calculations were correct.

 We tested our 
 machine by calculating those integers of length 8 that can be expressed 
 as the sum of up to 3 binary squares of length 4, and up to 4 binary squares of length 6.
 We then used the {\tt ULTIMATE} framework to test that those length-8 integers
 are accepted by our machine, but all others are rejected. The code running this
 test can be found as ${\tt Minus2Minus4SquareConjecture - Test 1}$ at 
 \codelink.

 We also tested the machine by calculating those integers of length 10
 that can be expressed as the sum of up to to 2 binary squares
 of length 6, and up to 4 binary squares of length 8. 
 We then built the analogous machine and confirmed that these
 length-10 integers are accepted, but all others are rejected. 
 We then repeated this test for those integers of length 10
 that can be expressed as the sum of up to to 3 binary squares
 of length 6, and up to 3 binary squares of length 8. 
The code running these
 tests can be found as ${\tt Minus2Minus4SquareConjecture - Test 2}$ 
 and ${\tt Minus2Minus4SquareConjecture - Test 3}$ at \codelink.


\section{Optimality}
In this section we show that the ``4'' in Theorem~\ref{main} is optimal.

\begin{theorem}
For $n \geq 1$, $n$ odd,
$n \not= 9$, the number $2^n$ is not the sum of three or
fewer (positive) binary squares.
\end{theorem}

\begin{proof}
Let $m \geq 0$ and
$n = 2m+1 $ be odd.  The cases $m = 0, 1,2, 3$ are easy to verify
by hand, so assume $m \geq 4$.  

In what follows we distinguish between ``mod'' used in the ordinary notion of 
congruence (where $x \equiv \modd{a} {b}$ means that $b$ divides $x-a$),
and the use of ``mod'' as a function, where $x = {a \bmod b}$ means
both that $x \equiv \modd{a} {b}$ and that $0 \leq a < b$.

Clearly $N := 2^n$ is not a binary square.

Suppose $N$ is the sum of two positive binary squares.
The largest binary square $< N$ is clearly $2^{2m} - 1$.
Hence the sum of two binary squares is either larger than $N$,
or no larger than $2(2^{2m} - 1) = 2^{2m+1} - 2 < N$, a contradiction.

The remaining case is that $2^{2m+1}$ is the sum of three binary
squares, say $N = A + B + C$ with 
$$A = a (2^e + 1) \geq B = b (2^f + 1) \geq C = c (2^g + 1)$$
with $e \geq f \geq g$ and $2^{e-1} \leq a < 2^e$,
$2^{f-1} \leq b < 2^f$, and
$2^{g-1} \leq c < 2^g$.  Clearly $1 \leq e, f, g \leq m$.

We first observe that $e = m$.  For otherwise, $e \leq m-1$ and the
inequality $e \geq f \geq g$ implies
$$N =  A+ B + C \leq  3(2^{m-1} - 1)2^{m-1} < 3 \cdot 2^{2m-2} < N,$$
a contradiction.

Similarly, we observe that $f = m$.  For otherwise
$$ N = A+ B +C \leq  (2^m - 1)2^m + 2(2^{m-1} -1)2^{m-1} 
< 3 \cdot 2^{2m-1} < N,$$
a contradiction.   

Thus, setting $d = a+b$, we see that
$N = d(2^m + 1) + c (2^g + 1)$ 
where $2^m \leq d \leq 2^{m+1} - 2$.  
Suppose $d = 2^{m+1}-2$.  Then 
$N = d(2^m + 1) + c (2^g + 1)$ 
implies that $C = c (2^g + 1) = 2$.  But $C = 2$ is not a binary square.
So in fact  $2^m \leq d  \leq 2^{m+1} - 3$.

Next we argue that $g > m/2$.  For otherwise $g \leq m/2$ and we have
$$N = d(2^m + 1) + c (2^g + 1) \leq (2^{m+1}-3)(2^m + 1) + (2^{m/2} - 1)(2^{m/2} + 1) = 2^{2m+1} - 4 = N-4,$$
a contradiction.

Next we argue that $g < m$.  For otherwise $g = m$ and then
$N = 2^{2m+1} =  A + B + C = (a+b+c)(2^m +1)$.  But then $2^{2m+1}$ is
divisible by the odd number $2^m + 1$, a contradiction.

Now consider the equation $N = d(2^m + 1) + c (2^g + 1)$ and take it
modulo $2^m + 1$.  We have
$2^{2m + 1} - 2 = 2 (2^m - 1)(2^m + 1)  \equiv \modd{0} {2^m + 1}$,
and so $N = 2^{2m+1} \equiv \modd{2} {2^m + 1}$.

Thus we get 
\begin{equation}
c (2^g + 1) \equiv \modd{2} {2^m + 1}.
\label{con}
\end{equation}
It suffices to show that the congruence \eqref{con} has no solutions
in the possible range for $c$,
except when $m = 4$ and $g = 3$.
In order to see this, we need a technical lemma.

\begin{lemma}
Suppose $m, g \geq 1$ are integers with
$m/2 < g < m$.  Suppose $c$ is an integer with
$2^{g-1} \leq c < 2^g$.  Using Euclidean division,
find the unique expression of
$c$ as $t\cdot 2^{m-g} + u$ for $0 \leq u < 2^{m-g}$.
Then
$$ c(2^g + 1) \bmod (2^m + 1) = t(2^{m-g} - 1) + u(2^g + 1) .$$
\end{lemma}

\begin{proof}
We have
\begin{align*}
c (2^g + 1) &= (t \cdot 2^{m-g} + u) (2^g + 1) \\
&= t \cdot 2^m + t \cdot 2^{m-g} + u(2^g + 1) \\
&= t (2^m + 1) + t (2^{m-g} - 1) + u (2^g + 1) \\
&\equiv \modd{ t (2^{m-g} - 1) + u (2^g + 1) } {2^m + 1} .
\end{align*}
This last congruence alone does not prove what we want; we also have to
show that 
$$0 \leq t(2^{m-g} - 1) + u(2^g + 1) < 2^m + 1$$
so that the
residues don't ``wrap around'' when computed modulo $2^m + 1$.
However, $t = \lfloor c/2^{m-g} \rfloor = 2^{2g-m} - 1$, and so
\begin{align*}
t(2^{m-g} - 1) + u(2^g + 1) & \leq 
(2^{2g-m} - 1)(2^{m-g} - 1) + (2^{m-g} - 1)(2^g + 1) \\
&= 2^m - 2^{2g-m} < 2^m + 1,
\end{align*}
as desired.  
\end{proof}

Now from the Lemma we see that the
expression $c (2^g + 1) \bmod (2^m + 1)$
achieves its smallest value when $ c= 2^{g-1} $ (for
then $t = 2^{2g-m-1}$ and $u = 0$), and this smallest value
is $2^{2g-m-1} (2^{m-g} - 1) > 2$, except when
$m = 4$, $g = 3$.  
\end{proof}

\begin{remark}
When $m = 4$ and $g = 3$, letting $c = 28$ and $d = 4$ we
get the solution $512 = 2^9 = 28 \cdot (2^4 + 1) + 4 \cdot (2^3 + 1)$.
This corresponds to two distinct expressions of $2^9$ as the sum of three
binary squares: $512 = 255 + 221 + 36$ and
$512 = 238 + 238 + 36$.
\end{remark}

\section{Other results}

Our technique can be used to obtain other results in additive number
theory.  For example, recently Crocker \cite{Crocker:2008} and Platt \& Trudgian \cite{Platt&Trudgian:2016}
studied the integers representable as the sum of two ordinary squares and two powers of $2$.  The analogue of this theorem is the following:

\begin{lemma}
\leavevmode
\begin{itemize}
\item[(a)] Every length-$n$ integer, $n$ odd, $n \geq 7$, is the sum of 
at most two powers of 2 and either:
\begin{itemize}
\item at most two binary squares of length $n-1$, or 
\item at most one binary square of length $n-1$ and one of length $n-3$.
\end{itemize}
\item[(b)] Every length-$n$ integer, $n$ even, $n \geq 10$, is the sum of 
at most two powers of 2 and either:
\begin{itemize}
\item at most one binary square of length $n$ and one of length $n-4$, or 
\item at most one binary square of length $n-2$ and one of length $n-4$.
\end{itemize}
\end{itemize}
\label{squarepower}
\end{lemma}
\begin{proof}
We use a similar proof strategy as before. The {\tt ULTIMATE}
code proving our result can be found in the files $\tt OddSquarePowerConjecture.ats$ 
and
$\tt EvenSquarePowerConjecture.ats$
at \codelink; \\ there one can also find the generators can be found as 
$\tt OddSquarePower.cpp$ and $\tt EvenSquarePower.cpp$.

The final machines for the odd-length and even-length cases have 806 and 2175  states
respectively. 
The language inclusion assertions all hold. This concludes the proof.
\end {proof}
We thus have the following theorem:
\begin{theorem}
Every natural number $N$ is the sum of at most two binary 
squares and at most two powers of 2.
\end{theorem}
\begin{proof}
For $N < 512$, the result can be easily verified.  Otherwise, we use
Lemma~\ref{squarepower} (a) if $N$ is an odd-length binary number and
Lemma~\ref{squarepower} (b) if it is even.
\end{proof}
We also consider the notion of {\it generalized binary squares}. 
A number $N$ is called a generalized binary square if one can concatenate
0 or more leading zeroes to its binary representation
to produce a binary square. 
As an example, $9$ is a generalized binary square, since $9$ in base 2
is $1001$, which can be written as $001001 = (001)(001)$.
The first few generalized binary squares are
$$ 0,3,5,9,10,15,17,18,27,33,34,36,45,51,54,63, \ldots;$$
they form sequence \seqnum{A175468} in the OEIS \cite{oeis}. 

In what follows, when we refer to the length of a generalized binary square,
we mean the length including the leading zeroes. Thus, $9$ is a 
generalized binary square of length $6$ (and not $4$).
\begin{lemma}
\leavevmode
\item[(a)] Every length-$n$ integer, $n \geq 7$, $n$ odd, is the sum of 3 generalized 
binary squares, of lengths $n+1$, $n-1$, and $n-3$.

\item[(b)] Every length-$n$ integer, $n \geq 8$, $n$ even, is the sum of 3 generalized 
binary squares, of lengths $n$, $n-2$, and $n-4$.
\label{gensquarepower}
\end{lemma}
\begin{proof}
We use a very similar proof strategy as in the proof of Lemma~\ref{aayush}. We drop
the requirement that the most significant digit of our guessed squares be 1, thus
allowing for generalized binary squares. Note that the square of length $n+1$ 
in part (a) must start with a 0.

The {\tt ULTIMATE}
code proving our result can be found in the files $\tt OddGenSquareConjecture.ats$ 
and
$\tt EvenGenSquareConjecture.ats$
at \codelink;  there one can also find
the generators 
$\tt OddGeneralizedSquares.cpp$ and $\tt EvenGeneralizedSquares.cpp$.
The final machines for the odd-length and even-length cases have 132 and 263 states
respectively.
\end{proof}
We thus have the following theorem:
\begin{theorem}
Every natural number $N > 7$ is the sum of 3 generalized binary squares.
\end{theorem}
\begin{proof}
For $7 < N < 64$ the result can be easily verified.  
Otherwise, we use
Lemma~\ref{gensquarepower} (a) is an odd-length binary number and
Lemma~\ref{gensquarepower} (b) if it is even.
\end{proof}

\section{Further work} 

Numerical evidence suggests the following two conjectures:

\begin{conjecture}
Let $\alpha_3$ denote the lower asymptotic density of the set $S_3$ of natural numbers that are the sum of three binary squares.  Then $\alpha_3 < 0.9$.  
\end{conjecture}
We could also focus on sums of {\it positive} binary squares.  (For the analogous problem dealing with ordinary squares, see, e.g., 
\cite[Chapter 6]{Grosswald:1985}.)  It seems likely that our method could be used to prove the following result.
\begin{conjecture}
Every natural number $>1772$ is the sum of exactly four positive
binary squares.  There are $112$ exceptions, given below:
\begin{multline*} 
0,1,2,3,4,5,6,7,8,9,10,11,13,14,15,16,17,18,20,21,22,23,25,27,28,29,30,32,34,35, \\
37,39,41,42,44,46,47,49,51,53,56,58,62,65,67,74,83,88,95,100,104,107,109,113,116,\\
122,125,131,134,140,143,148,149,155,158,160,161,167,170,173,175,182,184,368,385,\\
402,407,419,424,436,441,458,475,492,509,526,543,552,560,569,587,599,608,613,\\
620,625,638,647,653,671,686,698,713,1508,1541,1574,1607,1640,1673,1706,1739,1772 . 
\end{multline*}
\end{conjecture}
Other interesting things to investigate include estimating
the number
of distinct representations of $N$ as a sum of four binary
squares, both in the case where order matters, and where order does not matter.  These are sequences
\seqnum{A290335} and \seqnum{A298731} in the OEIS,
respectively.

In recent work \cite{Kane&Sanna&Shallit:2018} it was proved, using a combinatorial and number-theoretic approach, that the binary $k$'th powers form an asymptotic basis of finite order for the multiples of $\gcd(k, 2^k-1)$.  However, the constant obtained thereby is rather large.


\end{document}